\documentclass[letterpaper,12pt]{amsart}

\usepackage{amssymb}
\usepackage{amsmath,amscd}

\newtheorem{theorem}{Theorem}
\newtheorem{lemma}{Lemma}[section]
\newtheorem{proposition}[lemma]{Proposition}

\theoremstyle{definition}

\newtheorem{remark}[lemma]{Remark}
\newtheorem{notation}[lemma]{Notation}

\newcommand{\Gal}{\mathrm{Gal}}
\newcommand{\Hol}{\mathrm{Hol}}
\newcommand{\Hom}{\mathrm{Hom}}
\newcommand{\Aut}{\mathrm{Aut}}
\newcommand{\Perm}{\mathrm{Perm}}

\newcommand{\ord}{\mathrm{ord}}
\newcommand{\id}{\mathrm{id}}

\newcommand{\Z}{\mathbb{Z}}

\begin{document}
\title[Counting Hopf-Galois Structures]
{Counting Hopf-Galois Structures on Cyclic Field Extensions of
  Squarefree Degree}

\author{Ali A.~Alabdali}
 \thanks{The first-named author acknowledges support from The
    Higher Committee for Education Development in Iraq.}

\address{(A.~Alabdali) Department of Mathematics, College of Engineering,
  Mathematics and Physical Sciences, University of Exeter, Exeter 
EX4 4QF U.K.}

\address{(A.~Alabdali) Department of Mathematics, College of Education for Pure Science,
University of Mosul, Mosul, Iraq.}   
\email{aaab201@exeter.ac.uk}

\author{Nigel P.~Byott}
\address{(N.~Byott) Department of Mathematics, College of Engineering,
  Mathematics and Physical Sciences, University of Exeter, Exeter 
EX4 4QF U.K.}  
\email{N.P.Byott@exeter.ac.uk}

\date{\today}
\subjclass{12F10, 16T05}
\keywords{Hopf-Galois structures; field extensions; groups of squarefree
order}

\bibliographystyle{amsalpha}

\begin{abstract}
We investigate Hopf-Galois structures on a cyclic field extension
$L/K$ of squarefree degree $n$. By a result of Greither and Pareigis,
each such Hopf-Galois structure corresponds to a group of order $n$,
whose isomorphism class we call the type of the Hopf-Galois
structure. We show that every group of order $n$ can occur, and we
determine the number of Hopf-Galois structures of each type. We then
express the total number of Hopf-Galois structures on $L/K$ as a sum
over factorisations of $n$ into three parts. As examples, we give
closed expressions for the number of Hopf-Galois structures on a cyclic
extension whose degree is a product of three distinct primes. (There
are several cases, depending on congruence conditions between the
primes.) We also consider one case where the degree is a product
of four primes.
\end{abstract}

\maketitle

\section{Introduction and Statement of Results}

Let $L/K$ be a finite Galois extension of fields with Galois group
$\Gamma$. Then the group algebra $K[\Gamma]$ is a $K$-Hopf algebra,
and its action on $L$ endows $L/K$ with a Hopf-Galois structure. In
general, this is one among many possible Hopf-Galois structures on
$L/K$. Greither and Pareigis \cite{GP} showed that these Hopf-Galois
structures correspond to certain regular subgroups $G$ in the group
$\Perm(\Gamma)$ of permutations of the underlying set of
$\Gamma$. Finding all Hopf-Galois structures in any particular case
then becomes a combinatorial question in group theory. The groups
$\Gamma$ and $G$ necessarily have the same order, but need not be
isomorphic. We refer to the isomorphism type of $G$ as the type of the
corresponding Hopf-Galois structure.

There is a substantial literature on Hopf-Galois structures on various
classes of field extension. We mention a few results now, and some
others in the final section of this paper.  Let $p$ be an odd prime. A
cyclic field extension of degree $p^m$ admits precisely $p^{m-1}$
Hopf-Galois structures, all of cyclic type \cite{kohl}. An elementary
abelian extension of degree $p^m$ admits many more: there are at least
$p^{m(m-1)-1}(p-1)$ Hopf-Galois structures of elementary abelian type
if $p>m$ \cite{Ch05}, and there are also some of nonabelian type if $m
\geq 3$ \cite{BC}.  For a Galois extension whose Galois group $\Gamma$
is abelian, the type $G$ of any Hopf-Galois structure must be soluble
\cite{HGSsol}, although for a soluble, nonabelian Galois group
$\Gamma$ there can be Hopf-Galois structures whose type is not soluble.
Recently, Crespo, Rio and Vela \cite{CRV} have
investigated those Hopf-Galois structures on an extension $L/K$ which
arise by combining Hopf-Galois structures on $L/F$ and on $F/K$ for
some intermediate field $F$.

In this paper, we investigate Hopf-Galois structures on cyclic
extensions $L/K$ of arbitrary squarefree degree. Thus we consider
cyclic extensions whose degree has a prime factorisation at the other
extreme to those treated in \cite{kohl}. We intend to discuss
Hopf-Galois structures on arbitrary Galois extensions of squarefree
degree in a future paper.

The type of a Hopf-Galois structure on a cyclic extension of
squarefree degree $n$ could potentially be any group $G$ of
order $n$. There may be many of these.  Indeed, H\"older
\cite{holder} showed that the number of isomorphism types of groups of
squarefree order $n$ is given by
\begin{equation} \label{holder-form}
  \sum_{de=n} \prod_{p \mid d} \left(
  \frac{p^{v(p,e)}-1}{p-1} \right), 
\end{equation}
where the sum is over ordered pairs $(d,e)$ of positive integers such
that $de=n$, the product is over primes $p$ dividing $d$, and $v(p,e)$
is the number of distinct prime factors $q$ of $e$ with $q \equiv 1
\pmod{p}$. It is clear that, as $n$ varies over all squarefree
integers, the expression (\ref{holder-form}) can become arbitrarily
large.

It is an immediate consequence of Theorem \ref{HGS-count} below that,
for each group $G$ of order $n$, the number of Hopf-Galois structures
of type $G$ on a cyclic extension of degree $n$ cannot be zero. Thus
all possible types do in fact occur. The cyclic extensions of
squarefree degree therefore form a class for which both the number of
distinct types of Hopf-Galois structures on a given extension, and the
number of distinct prime factors of the degree of a given extension,
may be arbitrarily large. To the best of our knowledge, this is the
first class of extensions with these properties for which it has been
possible to enumerate all Hopf-Galois structures.  For comparison, we
mention that, when the Galois group $\Gamma$ is a nonabelian simple
group, the number of prime factors of $|\Gamma|$ may be arbitrarily
large, but there are only two Hopf-Galois structures, both of type
$\Gamma$ \cite{simple}. On the other hand, for Galois extensions of
degree $p_1 p_2 p_3$, where $p_1$, $p_2$, $p_3$ are distinct odd
primes satisfying certain congruence conditions, Kohl \cite{kohl-ANT}
has determined all Hopf-Galois structures for each possible Galois
group. In this case, the number of distinct types may be arbitrarily
large, but the number of primes dividing the degree is of course fixed
at three.

We will see in Proposition \ref{centre} that each group $G$ of
squarefree order $n$ gives rise to a factorisation $n=dgz$ of $n$, in
which $g$ (respectively, $z$) is the order of the commutator subgroup
$G'$ (respectively, the centre $Z(G)$) of $G$. We can now state the
first of our two main results.

\begin{theorem} \label{HGS-count}
Let $L/K$ be a cyclic extension of fields of squarefree degree $n$,
and let $G$ be any group of order $n$. Let $z=|Z(G)|$, $g=|G'|$ and
$d=n/(gz)$. Then $L/K$ admits precisely $2^{\omega(g)} \varphi(d)$
Hopf-Galois structures of type $G$, where $\varphi$ is Euler's totient
function and $\omega(g)$ is the number of
(distinct) prime factors of $g$.
\end{theorem}

Our second result gives the total number of Hopf-Galois structures.

\begin{theorem} \label{main}
The number of Hopf-Galois structures on a  cyclic field extension of
squarefree degree $n$ is 
\begin{equation} \label{HGS-formula}
   \sum_{dgz=n} 2^{\omega(g)} \mu(z) \prod_{p \mid d} \left(
  p^{v(p,g)}-1 \right), 
\end{equation}
where the product is over ordered triples $(d,g,z)$ of natural numbers
with $dgz=n$. Here $\mu$ is the M\"obius function.
\end{theorem}

We remark that (\ref{HGS-formula}) has a similar shape
to H\"{o}lder's formula (\ref{holder-form}), although with a sum over
factorisations into three parts rather than two. In both cases, the
term for each factorisation involves a product over primes $p$
dividing $d$, in which the contribution corresponding to $p$
does not depend on $d$ and $p$ alone. (In (\ref{holder-form}) it
depends on $e$, and in (\ref{HGS-formula}) on $g$.)

\section{Preliminaries on Hopf-Galois structures}

Let $L/K$ be a field extension of finite degree, and let $H$ be a
cocommutative $K$-Hopf algebra acting on $L$ . We write $\Delta \colon
H \to H \otimes_K H$ and $\epsilon \colon H \to K$ for the
comultiplication and counit maps on $H$, and use Sweedler's notation
$\Delta(h)=\sum_{(h)} h_{(1)} \otimes h_{(2)}$. We will say that the
action of $H$ on $L$ makes $L$ into an $H$-module algebra if $h \cdot
(xy) = \sum_{(h)} (h_{(1)} \cdot 
x) \otimes (h_{(2)}\cdot y)$ and $h \cdot k = \epsilon(h) k$ for all
$h \in H$, all $x$, $y \in L$ and all $k \in K$. A Hopf-Galois
structure on $L$ consists of a Hopf algebra $H$ acting on $L$ so that $L$ is an
$H$-module algebra and the $K$-linear map  
$\theta \colon L \otimes_K L \to \Hom_K(H,L)$ is bijective, where
$\theta(x\otimes y)(h)=x (h \cdot y)$ for $x$, $y \in L$
and $h \in H$. 

When $L/K$ is separable, Greither and Pareigis \cite{GP} used descent
theory to show how all Hopf-Galois structures on $L/K$ could be
described in group-theoretic terms. We consider here only the special case
where $L/K$ is a Galois extension in the classical sense (that is,
$L/K$ is normal as well as separable). Let $\Gamma=\Gal(L/K)$ be the
Galois group of $L/K$. Then the Hopf-Galois structures on $L/K$
correspond bijectively to subgroups $G$ of $\Perm(\Gamma)$ which are
regular on $\Gamma$ and are normalised by the group $\lambda(\Gamma)$
of left translations by elements of $\Gamma$. Recall that a group $G$
acting on a set $X$ is regular if the action is transitive on $X$ and the
stabiliser of any point is trivial.

The direct determination of all regular subgroups in $\Perm(\Gamma)$
normalised by $\lambda(\Gamma)$ is often difficult as the group
$\Perm(\Gamma)$ is large. However, the condition that
$\lambda(\Gamma)$ normalises $G$ means that $\Gamma$ is contained in
the holomorph $\Hol(G)=G \rtimes \Aut(G)$ of $G$, where the latter
group is viewed as a subgroup of $\Perm(\Gamma)$ and is usually
much smaller than $\Perm(\Gamma)$. We may then view $\Gamma$ as acting
on $G$, and this action is again regular. If the isomorphism types of
groups $G^*$ of order $|\Gamma|$ admit a manageable classification,
the Hopf-Galois structures on $L/K$ can be determined by
considering each $G^*$ in turn and finding the regular subgroups
$\Gamma^*$ of $\Hol(G^*) $ which are isomorphic to $\Gamma$. This
leads to the following result, cf.~ \cite[Cor.~to Prop.~1]{unique} or
\cite[\S7]{Ch00}:
\begin{lemma} \label{count-HGS}
Let $L/K$ be a finite Galois extension of fields with Galois group
$\Gamma$, and, for any group $G$ with $|G|=|\Gamma|$, let $e'(G,\Gamma)$
be the number of regular subgroups of $\Hol(G)$ isomorphic to
$\Gamma$. Then the number $e(G,\Gamma)$ of Hopf-Galois structures
on $L/K$ of type $G$ is given by
$$   e(G,\Gamma) = \frac{|\Aut(\Gamma)|}{|\Aut(G)|} \;  e'(G,\Gamma). $$
Moreover, the total number of Hopf-Galois structures on $L/K$ is given
by $\sum_{G} e(G,\Gamma)$, where the sum is over all
isomorphism types $G$ of groups of order $|\Gamma|$. 
\end{lemma}

\section{Preliminaries on groups of squarefree order}

We will call a finite group a $C$-group if all its Sylow subgroups are
cyclic. In particular, any group  of squarefree order is a
$C$-group. All $C$-groups are metabelian, so $C$-groups can in principle be
classified \cite[10.1.10]{Robinson}. This classification is given in a
rather explicit form in a paper of Murty and Murty \cite{MM}, who
investigated the asymptotic behaviour of the number of $C$-groups of
order up to a given bound. We state their classification result, in
the special case of groups of squarefree order, as Lemma
\ref{sf-class} below. 

\begin{notation}
For an integer $N \geq 1$, we denote by $\Z_N$ the ring $\Z/N\Z$ of
integers modulo $N$, and by $U(N)$ the group of units in $\Z_N$. We write
$\ord_N(a)$ for the order of an element $a \in U(N)$. Abusing
notation, we will often use the same symbol for an element of $\Z$ and
its class in $\Z_N$. We write $1_G$ for the identity element of a
group $G$. 
\end{notation}

\begin{lemma} \label{sf-class}
Let $n$ be squarefree. Then any group of order $n$ has the form
$$   G(d,e,k)= \langle \sigma, \tau \colon \sigma^e=\tau^d=1_G, \tau
\sigma \tau^{-1} = \sigma^k \rangle $$ 
where $n=de$, $\gcd(d,e)=1$ and $\ord_e(k)=d$. Conversely, any choice
of $d$, $e$ and $k$ satisfying these conditions gives a group
$G(d,e,k)$ of order $n$. Moreover, two such groups $G(d,e,k)$ and
$G(d',e',k')$ are isomorphic if and only if $d=d'$, $e=e'$, and
$k$, $k'$  generate the same cyclic subgroup of $U(e)$.
\end{lemma}
\begin{proof}
This follows from \cite[Lemmas 3.5 \& 3.6]{MM}.
\end{proof}

\begin{remark} \label{d-phi-e}
The existence of $k$ with $\ord_e(k) = d$ implies that $d$ divides
$\varphi(e)=|U(e)|$. Thus there may be many factorisations
$n=de$ of $n$ for which no groups $G(d,e,k)$ occur. 
\end{remark}

\begin{remark} \label{holder-rmk}
We note in passing how H\"older's formula (\ref{holder-form}) follows
from Lemma \ref{sf-class}. For fixed $d$ and $e$, the number of
isomorphism types of group $G(d,e,k)$ is the number of (necessarily cyclic)
subgroups of order $d$ in $U(e)$. Each such group is the
product of its Sylow $p$-subgroups for the primes $p$ dividing $d$.
For each such $p$, the $p$-rank of $U(e)$ is $v(p,e)$, so
$U(p)$ contains $(p^{v(p,e)}-1)/(p-1)$ subgroups of order
$p$. Taking the product over $p$ gives the number of subgroups of
order $d$.  Summing over $d$ yields the formula (\ref{holder-form})
for the number of isomorphism types of groups of order $n$.
\end{remark}

\begin{proposition} \label{centre}
Let $G=G(d,e,k)$ be a group of squarefree order $n$ as in Lemma
\ref{sf-class}. Let $z= \gcd(e,k-1)$ and $g=e/z$, so that we have 
factorisations $e=gz$ and $n=de=dgz$. Then the centre $Z(G)$ of $G$ is
the cyclic group $\langle \sigma^g \rangle$ of order $z$, and the
commutator subgroup $G'$ of $G$ is the the cyclic group $\langle
\sigma^z \rangle$ of order $g$. 
\end{proposition}
\begin{proof}
For $\gamma=\sigma^a \tau^b \in G$, we have $\sigma^{-1}
\gamma \sigma=\sigma^{a-1+k^b}\tau^b$. Since $\ord_e(k) =d$,
it follows that $\gamma$ commutes with $\sigma$ if and only if $d \mid
b$. But then $\gamma=\sigma^a$ and $\tau \gamma \tau^{-1}=\tau
\sigma^a \tau^{-1} = \sigma^{ak}$. Thus $\tau \gamma \tau^{-1}=\gamma$
precisely when $e \mid a(k-1)$, that is, when $g \mid a$. Hence
$Z(G)=\langle \sigma^g \rangle$. 

Turning to $G'$, we have $\tau \sigma \tau^{-1}
\sigma^{-1}=\sigma^{k-1}$. Thus $G'$ contains the normal subgroup
$\langle \sigma^{k-1} \rangle = \langle \sigma^z \rangle$ of
$G$. Equality holds since $G/\langle \sigma^{k-1} \rangle$ is
abelian.  
\end{proof}
 
We next find the number of isomorphism classes of groups $G$
corresponding to the factorisation $n=dgz$. 

\begin{proposition} \label{isom-g}
Let $n=dgz$ be squarefree. Then the number of isomorphism types of
groups $G$ of order $n$ with $|Z(G)|=z$ and $|G'|=g$ is 
\begin{equation} \label{isom-g-count}
  \varphi(d)^{-1} \sum_{f \mid g} \mu\left( \frac{g}{f} \right) 
         \prod_{p \mid d} (p^{v(p,f)}-1).  
\end{equation}
\end{proposition}
\begin{proof}
We keep $d$ and $e=n/d$ fixed. For each factor $g$ of $e$ let $m(g)$
be the number of isomorphism types of groups $G=G(d,e,k)$ (with $k$
varying) for which $|G'|=g$. We need to show that $m(g)$ is given by
the formula (\ref{isom-g-count}).

Let $m^*(g)$ be the number of groups
$G(d,e,k)$ for which $|G'|$ {\em divides} $g$. Then 
$$ m^*(g) = \sum_{f \mid g} m(f), $$
and so, by M\"obius inversion,
\begin{equation} \label{mobius}
     m(g) = \sum_{f \mid g} m^*(f) \mu \left( \frac{g}{f} \right). 
\end{equation}
The distinct isomorphism types of groups $G$ correspond to distinct
subgroups $D$ of order $d$ in $U(e) \cong \prod_{q \mid e}
U(q)$, where the product is over primes $q$ dividing $e$. Let
$f \mid e$. Then $|G'|$ divides $f$ precisely when $e/f$ divides
$|Z(G)|$, which occurs when $D$ has trivial projection in the factor
$U(q)$ for each prime $q$ dividing $e/f$. Hence $m^*(f)$ is the
number of subgroups of order $d$ in $U(f)$, and, arguing as in
Remark \ref{holder-rmk}, this is $\prod_{p \mid d}
(p^{v(p,f)}-1)/(p-1)$. Substituting into (\ref{mobius}) and noting
that $\prod_{p \mid d} (p-1) = \varphi(d)$, we obtain the expression
(\ref{isom-g-count}) for $m(g)$.
\end{proof} 

\section{Automorphisms and the Holomorph}

For this section and the next, we fix a group $G=G(d,e,k)$ of
squarefree order $n$. We keep the preceding notation, so $g=|G'|$,
$z=|Z(G)|$, and $n=de=dgz$. Our goal is to find the number of cyclic
subgroups of $\Hol(G)$ which are regular on $G$. By Lemma
\ref{count-HGS}, this will enable us to find the number of Hopf-Galois
structures of type $G$ on a cyclic extension of degree $n$. In this
section, we will describe $\Aut(G)$ and $\Hol(G)$. In \S5, we
determine all regular cyclic subgroups in $\Hol(G)$ and complete the
proof of Theorem \ref{HGS-count}. In \S\ref{proof-main}, we sum over
the different isomorphism types $G$ to prove Theorem \ref{main}.

Until the end of \S\ref{proof-main}, we shall systematically use the
notation $p$ for prime factors of $d$ and $q$ for prime factors of
$e$. Thus the primes $q$ are of two types: either $q \mid g$ or $q
\mid z$.

We begin by recording a formula which allows us to perform calculations in
$G$ itself. For integers $h$ and $j \geq 0$, we define
\begin{equation}  \label{def-S}
     S(h,j)= \sum_{i=0}^{j-1} h^i. 
\end{equation}
In particular, $S(h,0)=0$. A simple induction shows that, for any
$a \in \Z$,  
\begin{equation} \label{pow-sig-a-t}
  (\sigma^a \tau)^j = \sigma^{aS(k,j)} \tau^j.
\end{equation}

The next result describes the automorphisms of $G$.

\begin{lemma} \label{automs}
We have $|\Aut(G)|=g \varphi(e)$ and 
$$    \Aut(G) \cong C_g \rtimes U(e), $$
where $a \in U(e)$ acts on $C_g$ by $x \mapsto x^a$. (Note that in
general this action is not faithful.) 

Explicitly, $\Aut(G)$ is generated by the automorphism $\theta$ and
automorphisms $\phi_s$ for each $s \in U(e)$, where
\begin{equation} \label{def-theta}
   \theta(\sigma)=\sigma, \qquad \theta(\tau)=\sigma^z \tau, 
   \end{equation}
and
\begin{equation} \label{def-phi-s}
   \phi_s(\sigma)=\sigma^s, \qquad \phi_s(\tau)=\tau. 
\end{equation}
These automorphisms satisfy the relations
\begin{equation} \label{autom-rel} 
  \theta^g= \id_G, \qquad \phi_s \phi_t = \phi_{st}, \qquad \phi_s
  \theta \phi_s^{-1} = \theta^s.
\end{equation}
\end{lemma} 
\begin{proof}
We first verify the existence of the automorphisms $\theta$ and
$\phi_s$. Since $(\sigma^z \tau) \sigma (\sigma^z
\tau)^{-1} = \sigma^k$, (\ref{def-theta}) will determine an
automorphism $\theta$ provided that $\sigma^z \tau$ has order
$d$. By (\ref{pow-sig-a-t}), this will hold if $e \mid zS(k,d)$, that
is, if $g \mid S(k,d)$. But for each prime $q \mid g$, we have $k^d
\equiv 1 \not \equiv k \pmod{q}$, so that
$$ S(k,d) = \frac{k^d-1}{k-1} \equiv 0 \pmod{q}. $$ 
Thus $g \mid S(k,d)$, as required. This shows the existence of the
automorphism $\theta$. For $s \in U(e)$, the element $\sigma^s$ has
order $e$ and $\tau \sigma^s \tau^{-1}=(\sigma^s)^k$. It follows that
there is an automorphism $\phi_s$ as given in (\ref{def-phi-s}).

It is clear that $\theta$ has order $g$. The remaining
relations in (\ref{autom-rel}) are easily verified by checking them on
the generators $\sigma$, $\tau$ of $G$.

We have now shown that $\theta$ and the $\phi_s$ generate a subgroup
of $\Aut(G)$ isomorphic to $C_g \rtimes U(e)$. This subgroup
has order $g \varphi(e)$. It remains to show that there are no further
automorphisms. 

Let $\psi \in \Aut(G)$. As $\langle \sigma \rangle$ is a
characteristic subgroup of $G$, being the unique
subgroup of order $e$, we have $\psi(\sigma)=\sigma^s$ for
some $s \in U(e)$. Let $\psi(\tau)=\sigma^a \tau^b$ with $0 \leq
b<d$. Since $\psi$ must satisfy $\psi(\tau) \psi(\sigma)
\psi(\tau)^{-1}=\psi(\sigma)^k$, we have $\sigma^{sk^b} =
\sigma^{sk}$ and hence $b=1$. Thus, by (\ref{pow-sig-a-t}) again,
$$ \psi(\tau)^d = \sigma^{aS(k,d)}, $$
so that $a S(k,d) \equiv 0 \pmod{e}$. In particular, for each prime
$q$ dividing $z$, we have $q \mid aS(k,d)$. But $S(k,d) \equiv d \not
\equiv 0 \pmod{q}$ since $k \equiv 1 \pmod{q}$. Thus $q \mid a$. It
follows that $a=zc$ for some $c \in \Z$, so $\psi=\theta^c \phi_s$.
\end{proof}

We now consider the holomorph $\Hol(G) = G \rtimes \Aut(G)$ of $G$. We
write an element of this group as $[\alpha,\psi]$, where $\alpha \in
G$ and $\psi \in \Aut(G)$. The multiplication in $\Hol(G)$ is 
given by 
\begin{equation} \label{hol}
     [\alpha, \psi] [\alpha', \psi'] = [\alpha \psi(\alpha'),
  \psi \psi']. 
\end{equation}
(The subgroup $G$ in $\Hol(G)$ is therefore identified with the left
translations in $\Perm(G)$.) In view of Lemma \ref{automs}, an
arbitrary $x \in \Hol(G)$ can be written $x = [ \sigma^a \tau^b,
  \theta^c \phi_s]$, where $a \in \Z_e$, $s \in U(e)$, $b \in \Z_d$
and $c \in \Z_g$.  In Lemma \ref{x-pow} below, we will give a formula
for powers of $x$ in the special case $b=1$. We will then show in
Proposition \ref{b1} why this case is all we need. We first introduce
some further notation.

Define
\begin{equation} \label{def-T}
  T(k,s,j) = \sum_{h=0}^{j-1} S(s,h) k^{h-1} \mbox{ for } j \geq 1,
  \qquad  T(k,s,0)=0, 
\end{equation}
where $S(s,h)$ is given by (\ref{def-S}). Note that we then have
$T(k,s,1)=0$ and 
$$ T(k,s,j+1) = T(k,s,j) + k^{j-1} S(s,j) \mbox{ for } j \geq 0. $$ 

\begin{lemma} \label{x-pow}
Let $x=[\sigma^a \tau, \theta^c \phi_s]$. Then, for $j
\geq 0$, we have 
\begin{equation} \label{xj}
   x^j = [ \sigma^{A(j)} \tau^j, \theta^{cS(s,j)} \phi_{s^j}] 
\end{equation}
where $A(j) = aS(sk,j) + czkT(k,s,j)$. 
\end{lemma}
\begin{proof}
We argue by induction on $j$. When $j=0$, we have $S(s,0)=T(k,s,0)=0$
and $A(0)=0$, so (\ref{xj}) holds in this case.  Assuming (\ref{xj})
for $j$, we have from (\ref{hol}) that
\begin{eqnarray*}
  x^{j+1} & = &  [\sigma^{A(j)} \tau^j, \theta^{cS(s,j)}\phi_{s^j}]
  [\sigma^a \tau, \theta^c \phi_s] \\ 
              & = & [ \sigma^{A(j)} \tau^j (
    \theta^{cS(s,j)}\phi_{s^j}(\sigma^a \tau)),
    \theta^{cS(s,j)}\phi_{s^j} \theta^c \phi_s].  
\end{eqnarray*}
Thus, using (\ref{autom-rel}), the second component of $x^{j+1}$ is 
$$ \theta^{cS(s,j)} \phi_{s^j} \theta^c \phi_s = \theta^{cS(s,j)}
   \theta^{cs^j} \phi_{s^j} \phi_s = \theta^{cS(s,j+1)}
   \phi_{s^{j+1}}, $$ 
as required for (\ref{xj}). As for the first component of
$x^{j+1}$, since  
$$    \theta^{cS(s,j)}\phi_{s^j}(\sigma^a \tau) = \sigma^{as^j}
     \sigma^{czS(s,j)} \tau, $$ 
we have 
\begin{eqnarray*}
  \sigma^{A(j)} \tau^j ( \theta^{cS(s,j)}\phi_{s^j}(\sigma^a \tau)) & = & 
        \sigma^{A(j)} \tau^j    \sigma^{as^j}  \sigma^{czS(s,j) } \tau \\
           & = & \sigma^{A(j)} \sigma^{as^j k^j}
                 \sigma^{czS(s,j) k^j}  \tau^{j+1}. 
\end{eqnarray*}
We write this as $\sigma^{A'} \tau^{j+1}$, and calculate
\begin{eqnarray*}
  A' & = & A(j) + as^j k^j +czS(s,j) k^j \\
       & = & a(S(sk,j)+ (sk)^j)  + czk[T(k,s,j)+k^{j-1}S(s,j)] \\
       & = & aS(sk,j+1) + czk T(k,s,j+1) \\
        & = & A(j+1).
\end{eqnarray*}
Thus (\ref{xj}) holds with $j$ replaced by $j+1$. This completes the induction.
\end{proof}

\begin{proposition} \label{b1}
Let $C$ be a cyclic subgroup of $\Hol(G)$ which is regular on
$G$. Then $C$ is generated by some element 
$$ x =[\sigma^a \tau, \theta^c \phi_s ], $$
in which $\tau$ occurs with exponent $1$. In fact, $C$ contains
precisely $\varphi(e)$ generators of this type.  
\end{proposition}
\begin{proof}
For any $\psi \in \Aut(G)$ and arbitrary $\alpha=\sigma^a \tau^b \in
G$, we have $\psi(\alpha)=\sigma^{a'} \tau^b$ for some $a' \in \Z$.
This is clear from Lemma \ref{automs} as it holds for $\psi=\phi_s$
and $\psi=\theta$. It then follows from (\ref{hol}) that the function
$\Hol(G) \to \langle \tau \rangle$, given by $[\sigma^a \tau^b, \psi]
\mapsto \tau^b$, is a group homomorphism. (This is not automatic,
since the function $\Hol(G) \to G$ given by $[\sigma^a \tau^b, \psi]
\mapsto \sigma^a \tau^b$, is not in general a homomorphism.) In
particular, for any $x=[\sigma^a \tau^b, \theta^c \phi_s] \in \Hol(G)$
and any $j \geq 1$, we have $x^j = [\sigma^A \tau^{bj}, \psi]$ for
some $A \in \Z_e$ and some $\psi \in \Aut(G)$, both depending on $j$. The
permutation $x^j$ of $G$ takes $1_G$ to $\sigma^A \tau^{bj}$.
 
Now let $C$ be a regular cyclic subgroup of $\Hol(G)$, and let
$x=[\sigma^a \tau^b, \theta^c \phi_s]$ be a generator. Thus $x$ has
order $n$. Since $C$ is transitive on $G$, the elements $\sigma^A
\tau^{bj}$, as $j$ varies, must run through all elements of $G$.
In particular, $bj$ must run through all residue classes modulo $d$.  
Hence $\gcd(b,d)=1$, and there exists $f \geq 1$ with $bf \equiv
1 \pmod{d}$. Since $\gcd(d,e)=1$, we may further assume that
$\gcd(f,e)=1$. Then $\gcd(f,n)=1$, so that $x^f$ is also a generator
of $C$, and $x^f=[\sigma^{A'} \tau^{bf}, \psi']=[\sigma^{A'} \tau,
  \psi']$ for some $A'$ and $\psi'$. Replacing $x$ by $x^f$, we
have found a generator of $C$ with $b=1$, as required.  

Now let $x$ be any such generator. Then $x^j$ will be
another if and only if $\gcd(j,n)=1$ and $j \equiv 1
\pmod{d}$. The number of such generators is therefore
$\varphi(n)/\varphi(d)=\varphi(e)$.
\end{proof}

\section{Hopf-Galois structures of type $G$} \label{type-G}

As a first step towards determining when the element $x$ in
Proposition \ref{b1} generates a regular subgroup, we give a condition
for transitivity. 

\begin{lemma} \label{reg-crit}
Let $x=[\sigma^a\tau, \theta^c \phi_s] \in \Hol(G)$. Then the subgroup
$\langle x \rangle$ of $\Hol(G)$ acts transitively on $G$ if and only if
$\langle x^d \rangle$ acts transitively on $\langle \sigma \rangle$.
\end{lemma}
\begin{proof}
Let $\langle x \rangle$ be transitive on $G$. Then, for each $i \in
\Z$, there is some $j$ such $x^j \cdot 1_G = \sigma^i$. Then $d \mid
j$ by (\ref{xj}). Thus $\langle x^d \rangle$ acts transitively on
$\langle \sigma \rangle$.  Conversely, suppose that $\langle x^d
\rangle$ acts transitively on $\langle \sigma \rangle$. Let $\sigma^i
\tau^j \in G$. By Lemma \ref{x-pow}, we have $x^{-j} \cdot \sigma^i
\tau^j \in \langle \sigma \rangle$. As $\langle x^d \rangle$ is
transitive on $\langle \sigma \rangle$, there is some $h \in \Z$ with
$x^{dh-j} \cdot \sigma^i \tau^j =1_G$. Thus the arbitrary element
$\sigma^i \tau^j$ lies in the same orbit under $\langle x \rangle$ as
$1_G$, so that $\langle x \rangle$ is transitive on $G$.
\end{proof}

In order to study the orbits of $\langle x^d \rangle$ on $\langle
\sigma \rangle$, we examine the congruence properties of the sums $S(k,j)$
and $T(k,s,j)$ defined in (\ref{def-S}) and (\ref{def-T}) when $j$ is
a multiple of $d$. 

\begin{proposition} \label{ST-cong}
Let $q$ be a prime dividing $e$. In the following, all congruences are
modulo $q$. We will omit the modulus for brevity. Abusing notation,
we will write $\frac{u}{v}$ in such a congruence to denote $uv^*$
where $vv^* \equiv 1$. (This notation will only be used when $v \not
\equiv  0$.) 

\begin{itemize}
\item[(i)] For any $s$, $i \in \Z$ with $i \geq 0$, we have 
$$   S(s,di) \equiv \begin{cases} 
   di & \mbox{if } s \equiv 1; \cr \cr
   \displaystyle{\frac{s^{di}-1}{s-1}} &
   \mbox{otherwise}. \end{cases} $$
\item[(ii)] Recall that $k^d \equiv 1$. If also $k \not \equiv 1$
  then, for any $s$, $i \in \Z$ with $i \geq 0$, we have
$$ T(k,s,di) \equiv \begin{cases}
    \displaystyle{\frac{di}{k(k-1)}} & \mbox{if } s \equiv 1; \cr \cr  
    \displaystyle{\frac{di}{k(s-1)}} & \mbox{if } sk \equiv 1; \cr \cr
   \displaystyle{\frac{ (s^{di}-1)}{k(s-1)(sk-1)}} &
   \mbox{otherwise}. \end{cases}  $$ 
\end{itemize}
\end{proposition}
\begin{proof}
(i) This is immediate.

(ii) The case $i=0$ is clear, so assume $i \geq 1$. First let $s
  \equiv 1$. Then $S(s,j) \equiv j$, so 
\begin{eqnarray*}
   (k-1)T(k,s,di) 
    & = & \sum_{j=0}^{di-1} (k-1)j  k^{j-1} \\
    & = & \sum_{j=0}^{di-1} j k^j -  \sum_{j=1}^{di-1}j k^{j-1} \\
   & = & \sum_{j=0}^{di-1} j k^j - \sum_{j=0}^{di-2} (j+1) k^j \\
   & = & \sum_{j=0}^{di-1} j k^j - \sum_{j=0}^{di-1} (j+1) k^j + dik^{di-1} \\
   & = & - \sum_{j=0}^{di-1} k^j +dik^{di-1} .
\end{eqnarray*}
As $k^d \equiv 1 \not \equiv k$, we then have
$$ (k-1) T(k,s,di) \equiv dik^{di-1} , $$
giving the result for $s \equiv 1$.

If $s \not \equiv 1$ then
\begin{eqnarray*} 
   T(k,s,di) 
 & \equiv &  \sum_{j=0}^{di-1} \left( \frac{s^j-1}{s-1} \right) k^{j-1} \\
 &  \equiv & \frac{1}{k(s-1)} \left[ \sum_{j=0}^{di-1} (sk)^j -
    \sum_{j=0}^{di-1} k^j \right]. 
\end{eqnarray*}
The second sum vanishes mod $q$. The first is congruent to $di$ if $sk
\equiv 1$, giving the result in this case. Finally if $sk \not \equiv
1 \not \equiv s$ then
\begin{eqnarray*}
    T(k,s,di)
  & \equiv & \frac{1}{k(s-1)}  \sum_{j=0}^{di-1} (sk)^j  \\
   &  \equiv &  \frac{1}{k(s-1)(sk-1)} \left( (sk)^{di}-1 \right)  \\
 & \equiv & \frac{1}{k(s-1)(sk-1)} \left( s^{di}-1 \right).
\end{eqnarray*}
\end{proof}

\begin{lemma}  \label{reg-conds}
Let $x = [\sigma^a \tau, \theta^c \phi_s ] \in \Hol(G)$, so
$a \in \Z_e$, $c \in \Z_g$, $s \in U(e)$. Then 
$x$ generates a regular cyclic subgroup of $\Hol(G)$ if and only if the
triple $(s,a,c)$ satisfies the following conditions: 
\begin{itemize}
\item[(i)] for each prime $q \mid z$, we have $s \equiv 1 \pmod{q}$ and
  $q \nmid a$;  
\item[(ii)]  for each prime $q \mid g$, either
$$ s \equiv 1 \pmod{q} \mbox{ and } c \not \equiv 0 \pmod{q}, 
   \mbox{ or } $$
$$ s\equiv k^{-1} \pmod{q} \mbox{ and } (s-1)a+cz  \not \equiv 0 \pmod{q}. $$
\end{itemize}
\end{lemma}
\begin{proof}
Suppose that $\langle x \rangle$ is regular, and hence transitive, on
$G$. By Lemma \ref{reg-crit}, $\langle x^d \rangle$ is transitive on
$\langle \sigma \rangle$.  It follows using Lemma \ref{x-pow} that the
expression
$$ A(di) = aS(sk,di)+czkT(k,s,di) $$ represents all residue classes
mod $e$ as $i$ varies.  In particular, $A(di)$ represents every
residue class mod $q$ for each prime factor $q$ of $e$.  We
investigate this condition for each $q$ in turn. Again, we omit
the modulus in congruences modulo $q$. 

First, let $q \mid z$, so $k \equiv 1$. 
If $s \not \equiv 1$, then $sk \not \equiv 1$, so by
Proposition \ref{ST-cong}(i) we have 
$$ A(di) \equiv \frac{a(s^{di}-1)}{s-1}, $$ 
which cannot represent all residue classes mod $q$ since there is no
$i$ such that $s^{di} \equiv 0$. On the other hand, if $s
\equiv 1$ then 
\begin{equation} \label{k-1} 
  A(di)  \equiv adi.
\end{equation}
Since $q \nmid d$, this represents all residue classes mod $q$
precisely when $q \nmid a$. Thus (i) holds.

Now let $q \mid g$, so $k \not \equiv 1$ but $k^d \equiv 1$. If $s
\not \equiv 1$ and $s \not \equiv k^{-1}$, then, using both parts of
Proposition \ref{ST-cong}, we have
\begin{eqnarray*}  
 A(di)  & \equiv &  \left( \frac{(sk)^{di}-1}{sk-1} \right) a + czk
 \left(\frac{s^{di}-1}{k(s-1)(sk-1)}\right) \\
 & = & \frac{(s^{di}-1)}{(s-1)(sk-1)}\left((s-1) a + cz\right). 
\end{eqnarray*}
Again, this cannot represent all residue classes mod $q$ since
$s^{di} \not \equiv 0$. 

It remains to consider the two
special cases $s \equiv 1 \not \equiv k$ and $s \equiv
k^{-1}\not \equiv 1$. 

If $s \equiv 1 \not \equiv k$ then, as $(sk)^d \equiv k^d
\equiv 1$, we have 
\begin{equation} \label{s-1}
   A(di) \equiv  czk \; \left(\frac{di}{k(k-1)} \right) = \frac{czdi}{k-1}.
\end{equation}
As $q \nmid zd$, this represents all residue classes mod $q$
precisely when $q \nmid c$, giving the first case in (ii).

If $s \equiv k^{-1} \not \equiv 1$ then
\begin{eqnarray} 
  A(di) & \equiv &  adi +  czk \left(\frac{di}{k(s-1)} \right)
         \nonumber \\
     &  \equiv & \frac{di}{s-1} \; \left( (s-1)a+cz \right).
                 \label{sk-1}     
\end{eqnarray}
This represents all residue classes mod $q$ precisely when $(s-1)a+cz
\not \equiv 0$, giving the second case in (ii).

We have now shown that if $x$ generates a regular cyclic subgroup,
then (i) and (ii) hold. 

Conversely, suppose that (i) and (ii) hold. Then, by Proposition
\ref{ST-cong}, the congruences (\ref{k-1}), (\ref{s-1}) and
(\ref{sk-1}) hold modulo all relevant $q$. For each prime $q \mid e$,
we then have that $A(di)$ represents all residue classes mod $q$ as
$i$ runs through any complete set of residues mod $q$. By the Chinese
Remainder Theorem, $A(di)$ then ranges through all residue classes mod
$e$ as $i$ does. By Lemma \ref{x-pow}, $\langle x^d \rangle$ is then
transitive on $\langle \sigma \rangle$, so $\langle x \rangle$ is
transitive on $G$ by Lemma \ref{reg-crit}. Finally, (\ref{k-1}),
(\ref{s-1}) and (\ref{sk-1}) show that $A(de) \equiv 0 \pmod{e}$, so that
$x^n = 1_G$. Hence $\langle x \rangle$ is regular on $G$.
\end{proof}

\begin{proof}[Proof of Theorem \ref{HGS-count}]
By Proposition \ref{b1}, any regular cyclic subgroup of $\Hol(G)$ is
generated by an element $x$ as in Lemma \ref{reg-conds}. We count the
number of triples $(s,a,c)$ satisfying the conditions there. As there
is only one possibility for $s \bmod q$ when $q \mid z$ but two when
$q \mid g$, there are $2^{\omega(g)}$ possibilities for $s \pmod
e$. Let us fix $s$ and consider the possibilities for $a$ and $c$. For
each prime $q \mid z$, condition (i) in Lemma \ref{reg-conds} excludes
one possibility for $a \bmod q$. For each $q \mid g$, we may choose $a
\bmod q$ arbitrarily, and then, in either case of condition (ii), one
possibility for $c \bmod q$ is excluded. Thus we have $\varphi(z)g$
choices for $a \bmod e$, and then $\varphi(g)$ choices for $c \bmod
g$.  The number of elements $x=[\sigma^a \tau, \theta^c \phi_s]$ which
generate a regular subgroup is therefore $2^{\omega(g)} \varphi(z)g
\varphi(g)$. By Proposition \ref{b1}, each regular cyclic subgroup
contains $\varphi(e)=\varphi(z)\varphi(g)$ such generators, so there
are $2^{\omega(g)}g$ of these subgroups. Thus, using Lemma
\ref{count-HGS} and Lemma \ref{automs} (and writing $C_n$ for the
cyclic group of order $n$), we find that the number of Hopf-Galois
structures of type $G$ is
$$ \frac{|\Aut(C_n)|}{|\Aut(G)|} \, 2^{\omega(g)} g = 
   \frac{\varphi(n)}{g\varphi(e)} 2^{\omega(g)} g = 2^{\omega(g)}
   \varphi(d). $$   
\end{proof}

\section{Proof of Theorem \ref{main}} \label{proof-main}

In this section, we obtain the total number of Hopf-Galois structures
on a cyclic field extension of squarefree degree $n$, thereby
completing the proof of Theorem \ref{main}. 

For each factorisation $n=dgz$, we have seen in Proposition
\ref{isom-g} that the number of corresponding isomorphism types of
group $G$ is  
$$ \varphi(d)^{-1} \sum_{f \mid g} \mu\left( \frac{g}{f} \right) 
\prod_{p \mid d} (p^{v(p,f)}-1).  $$ We have also seen in Theorem
\ref{HGS-count} that there $2^{\omega(g)} \varphi(d)$ Hopf-Galois
structures of each of these types. To obtain the total number of
Hopf-Galois structures, we simply multiply these two quantities and
sum over factorisations of $n$. This yields
$$ \sum_{dgz=n} 2^{\omega(g)} \varphi(d) \left( \varphi(d)^{-1} 
   \sum_{f \mid g} \mu\left( \frac{g}{f} \right) 
         \prod_{p \mid d} (p^{v(p,f)}-1) \right) . $$
Setting $t=g/f$ and noting that $\omega(g)=\omega(t)+\omega(f)$, we
can rewrite the previous sum as  
$$ \sum_{dftz=n} \mu(t) 2^{\omega(t)} 2^{\omega(f)} 
         \prod_{p \mid d} (p^{v(p,f)}-1). $$
Let $m=tz$, and observe that $\mu(t)=(-1)^{\omega(t)}$. The
sum then becomes
$$  \sum_{dfm=n} \left(\sum_{t \mid m} (-2)^{\omega(t)}\right) 2^{\omega(f)} 
         \prod_{p \mid d} (p^{v(p,f)}-1). $$
Recall that a function $F$ on the natural numbers is said to be
multiplicative if $F(rs)=F(r)F(s)$ whenever $\gcd(r,s)=1$. The
function $t \mapsto (-2)^{\omega(t)}$ is clearly multiplicative, and
hence so is the function $m \mapsto \sum_{t \mid m}
(-2)^{\omega(t)}$. However, evaluating this last function at a prime
$q$ gives $(-2)^{\omega(1)}+(-2)^{\omega(q)} = 1 + (-2) =-1=\mu(q)$. As
$\mu$ is also multiplicative, it follows
that $\sum_{t \mid m} (-2)^{\omega(t)} = \mu(m)$ for
squarefree $m$. (This is not true for arbitrary natural numbers $m$.)
Hence the total number of Hopf-Galois 
structures on a cyclic extension of squarefree degree $n$ is 
$$ \sum_{dfm=n} \mu(m) 2^{\omega(f)} 
         \prod_{p \mid d} (p^{v(p,f)}-1). $$
After a change of notation, this gives the formula
(\ref{HGS-formula}), completing the proof of Theorem \ref{main}.

\section{Examples}
In this section, we give some examples and show how several results in
the literature can be obtained as special cases of our
results. Throughout, $n$ is a squarefree integer and $L/K$ is a cyclic
Galois extension of degree $n$.

\subsection{Cyclic Hopf-Galois Structures} \label{cyclic-ex}
The group $G(d,e,k)$ in Lemma \ref{sf-class} is cyclic only when
$d=1$, $e=n$. Indeed, this is the only case where $G(d,e,k)$ is
abelian, or even nilpotent. In this case $z=n$, $g=1$, and Theorem
\ref{HGS-count} shows that there is only one Hopf-Galois structure of
cyclic (or abelian, or nilpotent) type on $L/K$. This can also be seen
from \cite[Theorem 2]{nilp}. The unique cyclic Hopf-Galois structure
is of course the classical one.

When $\gcd(n, \varphi(n))=1$, there are no other groups
$G(d,e,k)$, and hence there are no Hopf-Galois structures on $L/K$
beyond the classical one. This was shown, together with its converse,
in \cite{unique}. The case that $n$ is prime occurs in \cite{Ch89}.

\subsection{Dihedral Hopf-Galois Structures} \label{dihedral}
Let $n=2m$ where $m$ is an odd squarefree
number. The group $G(d,e,k)$ in Lemma \ref{sf-class} is dihedral when
$d=2$ and $k = -1 \in U(e)$. Then $e=g=m$. It follows from
Theorem \ref{HGS-count} that a cyclic extension of degree $n$ admits
$2^{\omega(m)}$ Hopf-Galois structures of dihedral type.

\subsection{Two Primes}
Let $n=pq$ for primes $p>q$. We assume that $q \mid (p-1)$ (since
otherwise the only group of order $n$ is the cyclic group $C_n$).  Up
to isomorphism, there are two groups of order $n$, the cyclic group
$C_n$ (with $d=1$, $g=1$, $e=z=pq$) and the metabelian group $M=C_p
\rtimes C_q$ where $C_q$ acts nontrivially on $C_p$ (so $d=q$,
$e=g=p$, $z=1$). As we have seen in \S\ref{cyclic-ex}, a cyclic
extension of degree $n$ admits just one Hopf-Galois structure of
cyclic type (namely the classical one). By Theorem \ref{HGS-count}, it
also admits $2^{\omega(p)}\varphi(q)=2(q-1)$ Hopf-Galois structures of
type $M$. This result was obtained in \cite{pq}, where we also
considered Hopf-Galois structures on a Galois extension with Galois
group $M$. When $q=2$, the result follows from \S\ref{dihedral}.

\begin{table} 
\centerline{ 
\begin{tabular}{|c|c|c|c|} \hline
 $d$ & $g$ & $z$ & Term in (\ref{HGS-formula}) \\ \hline
 $1$ & $pq$ & $1$ & $4$ \\
 $1$ & $p$ & $q$ & $-2$ \\
 $1$ & $q$ & $p$ & $-2$ \\
 $1$ & $1$ & $pq$ & $1$ \\
 $q$ & $p$ & $1$ & $2(q-1)$ \\ \hline
\end{tabular}
}  
\vskip5mm

\caption{Nonzero terms in (\ref{HGS-formula}) for $n=pq$.} 
\label{pq-table} 
\end{table}

Let us also verify that Proposition \ref{isom-g} correctly counts the
isomorphism types corresponding to each factorisation $n=dgz$, and
that Theorem \ref{main} correctly counts the total number of
Hopf-Galois structures, in this case. 

In the sum (\ref{isom-g-count}) of Proposition \ref{isom-g},
$\prod_{p \mid d} (p^{v(p,f)}-1)$ vanishes unless $d=1$ or $d=q$,
$f=p$ (so that also $g=p$). When $d=1$, (\ref{isom-g-count}) reduces
to
$$ \sum_{f \mid g} \mu\left(\frac{g}{f}\right) = 
   \begin{cases} 1 & \mbox{if } g=1; \\
                 0 & \mbox{otherwise.} \end{cases} $$ 
Thus when $d=1$, to get a group $G(d,e,k)$ we must take $g=1$ and $z=pq$.
We then have $G(d,e,k)=C_n$. When $d \neq 1$, all terms in
(\ref{isom-g-count}) vanish unless $d=q$, $g=p$, when the term for
$f=p$ gives $\varphi(q)^{-1} (q^1-1)=1$. Thus (\ref{isom-g-count})
tells us that there is just one isomorphism type of nonabelian group
of order $n$. Hence Proposition \ref{isom-g} does indeed give the
correct number of isomorphism types for each factorisation. By similar
reasoning (which we leave to the reader), H\"older's formula
(\ref{holder-form}) correctly predicts two isomorphism classes of
groups of order $n$.

We now turn to Theorem \ref{main}. The product over $p \mid d$
vanishes unless $d=1$ or $d=q$, $g=p$. The nonzero contributions to
(\ref{HGS-formula}) for the various factorisations $n=gzd$ are shown
in Table \ref{pq-table}. Summing the final column of Table
\ref{pq-table} gives the correct count of $2(q-1)+1$ Hopf-Galois
structures on a cyclic extension of degree $n=pq$. Table
\ref{pq-table} also illustrates an important feature of the formula
(\ref{HGS-formula}): factorisations $n=dgz$ for which there are no
corresponding groups $G$ can nevertheless contribute nonzero terms to
(\ref{HGS-formula}).

\subsection{Three Primes} \label{3-primes}

Let $n=p_1 p_2 p_3$ where $p_1 < p_2 <p_3$ are primes. Both the number
of isomorphism classes of groups of order $n$, and the number of
Hopf-Galois structures on a cyclic extension of degree $n$, depend on
congruence conditions relating the three primes. There are two
combinations of these conditions for which the Hopf-Galois structures
on all Galois extensions of degree $p_1 p_2 p_3$ (not just cyclic
extensions) have been enumerated. 

The first of these is when $p_1=2$ and $p_3=2p_2+1$ (so $p_2$ is a
Sophie Germain prime and $p_3$ is a safeprime). Kohl \cite[Theorem
  5.1]{kohl-mp} treated this case as an application of his method for
studying Hopf-Galois structures on Galois extensions of degree $mp$
(with $p$ prime and $m<p$). Those extensions with Galois group
$\Hol(C_{p_3})=C_{p_3} \rtimes C_{p_3-1}$ had previously been
considered in \cite{Ch03}.

The second situation where all Hopf-Galois structures have been
determined is when $p_1>2$ and $p_2 \equiv p_3 \equiv 1 \pmod{p_1}$
but $p_3 \not \equiv 1 \pmod{p_2}$. This case is treated in \cite[Theorem
  2.4]{kohl-ANT}. The same techniques could be applied to other
combinations of congruence conditions, but separate calculations would
be required for each case.

In the following, we will apply Theorem \ref{HGS-count} to count the
Hopf-Galois structures only on a {\em cyclic} extension of degree
$n=p_1 p_2 p_3$, but under all possible combinations of congruence
conditions. In particular, this will recover those parts of Kohl's
results in \cite{kohl-mp, kohl-ANT} which relate to cyclic extensions.

In Table \ref{3-prime-table} we show the factorisations $n=dgz$ for
which groups exist, the number of isomorphism types of these groups,
and the number of Hopf-Galois structures of each isomorphism type.

The first column of Table \ref{3-prime-table} numbers the
factorisations for ease of reference, and the factorisation is shown
in the next $3$ columns. The $5$th column shows the congruence
conditions which must be satisfied for groups to exist. The $6$th
column shows the number of isomorphism types of group corresponding to
the given factorisation, as given by Proposition \ref{isom-g}. These
can also be found directly, as explained below. The final column
shows the number of Hopf-Galois structures for each isomorphism
type. This is given by the formula $2^{\omega(g)} \varphi(d)$ of Theorem
\ref{HGS-count}.

We now explain how to find the values in the $6$th column of Table
\ref{3-prime-table} directly. (This illustrates in simple cases the
proof of Proposition \ref{isom-g}.) Consider for example case $2$,
where $d=p_1$, $g=p_2$, $z=p_3$, so $e=p_2 p_3$. The distinct
isomorphism types of groups $G(d,e,k)$ with these parameters
correspond to subgroups $\langle k \rangle \subseteq U(p_2 p_3)$ of
order $p_1$ for which $z=\gcd(k-1,p_2 p_3)=p_3$. Since $k \equiv 1
\pmod{p_3}$, we can identify $\langle k \rangle$ with a subgroup of
order $p_1$ in $U(p_2)$. Such a subgroup exists since $p_2 \equiv 1
\pmod{p_1}$, and it is unique since $U(p_2)$ is cyclic. Thus there is
just one group $G(d,e,k)$ in case 2. In case 4, however, where
$d=p_1$, $g=p_2 p_3$ and $z=1$, the isomorphism types of groups
$G(d,e,k)$ correspond to subgroups $\langle k \rangle \subseteq U(p_2
p_3)$ of order $p_1$ with $\gcd(k-1,p_2 p_3)=1$. Now $U(p_2 p_3) \cong
U(p_2) \times U(p_3)$ contains $p_1+1$ subgroups of order $p_1$.  For
one of these, $\gcd(k-1,p_2 p_3)=p_3$. This gives the group $G$ just
found in case 2. Another of the subgroups has $\gcd(k-1,p_2 p_3)=p_2$,
and this is counted in case 3. The remaining $p_1-1$ subgroups of
$U(p_2 p_3)$ give groups $G$ with $g=p_2p_3$ and $z=1$. Thus the
number of groups recorded in case $4$ is $p_1-1$.

\begin{table} 
\centerline{ 
\begin{tabular}{|c|ccc|c|c|c|} \hline
 Case & $d$ & $g$ & $z$ & Condition & \# groups & \# HGS per group \\ \hline
 $1$ & $1$ & $1$ & $p_1 p_2 p_3$ & & $1$ & $1$ \\
 $2$ & $p_1$ & $p_2$ & $p_3$ & $p_2 \equiv 1 \pmod{p_1}$ & $1$ &
 $2(p_1-1)$ \\
 $3$ & $p_1$ & $p_3$ & $p_2$ & $p_3 \equiv 1 \pmod{p_1}$ & $1$ &
 $2(p_1-1)$ \\
 $4$ & $p_1$ & $p_2 p_3$ & $1$ & $p_2 \equiv p_3 \equiv 1 \pmod{p_1}$ &
 $p_1-1$ & $4(p_1-1)$ \\ 
 $5$ & $p_2$ & $p_3$ & $p_1$ & $p_3 \equiv 1 \pmod{p_2}$ & $1$ &
 $2(p_2-1)$ \\ 
 $6$ & $p_1 p_2$ & $p_3$ & $1$ & $p_3  \equiv 1 \pmod{p_1 p_2}$ &
 $1$ & $2(p_1-1)(p_2-1)$ \\  \hline
\end{tabular}
}  
\vskip5mm

\caption{Numbers of isomorphism types and Hopf-Galois structures for
  $n=p_1 p_2 p_3$.}  \label{3-prime-table} 
\end{table}   

We now find the total number of Hopf-Galois structures on a cyclic
extension of degree $n$, treating each combination of relevant
congruence conditions on $p_1$, $p_2$, $p_3$ separately. The results
are shown in Table \ref{3-prime-total}. For each combination of
congruence conditions, we pick out the cases from Table
\ref{3-prime-table} where any groups $G(d,e,k)$ exist. To obtain the
total number of isomorphism types of groups of order $n$, we add the
numbers of groups from the corresponding rows in Table
\ref{3-prime-table}, giving the entries in the $5$th column of Table
\ref{3-prime-total}. These agree with the values given by Kohl
\cite[p.\ 46]{kohl-ANT}. To obtain the total number of Hopf-Galois
structures, we multiply the entries in the final two columns of Table
\ref{3-prime-table} and add these values for the appropriate
rows. After simplification, this gives the entries in the final
column of Table \ref{3-prime-total}. 

\begin{table} 
\centerline{ 
\begin{tabular}{|c|c|c|c|c|c|} \hline
 $p_2 \mid (p_3-1)$ & $p_1 \mid (p_3-1)$ &  $p_1 \mid (p_2 -1)$ &
  Cases & \# groups & Total \# HGS \\ \hline  
   no & no & no & $1$ & $1$ & $1$ \\
   no & no & yes & $1$, $2$ & $2$ & $2p_1-1$ \\
   no & yes & no & $1$, $3$ & $2$ & $2p_1-1$ \\
   no & yes & yes & $1$, $2$, $3$, $4$ & $p_1+2$ & $(2p_1-1)^2$ \\
   yes & no & no & $1$, $5$ & $2$ & $2p_2-1$ \\
   yes & no & yes & $1$, $2$, $5$ & $3$ & $2p_1 +2p_2-3$ \\
   yes & yes & no & $1$, $3$, $5$, $6$ & $4$ & $2p_1 p_2-1$ \\
   yes & yes & yes & $1$, $2$, $3$, $4$, $5$, $6$ & $p_1 + 4$ & 
          $4p_1^2 +2p_1 p_2 -6p_1 +1$ \\ \hline
\end{tabular}
}  
\vskip5mm

\caption{Total numbers of Hopf-Galois structures for $n=p_1 p_2 p_3$.}
\label{3-prime-total} 
\end{table}

We now specialise to the two situations considered in \cite[Theorem
  5.1]{kohl-mp} and \cite[Theorem 2.4]{kohl-ANT} in order to confirm
that we recover those parts of Kohl's results pertaining to cyclic
extensions.

First let $p_1=2$ and let $p_3=2p_2+1$ be a safeprime. Thus we have
$p_j \equiv 1 \pmod{p_i}$ whenever $1 \leq i \leq j \leq 3$,
corresponding to the final row (``yes--yes--yes'') in our Table
\ref{3-prime-total}. The first row of the table in \cite[Theorem
  5.1]{kohl-mp} shows that there are $6$ isomorphism types of groups
of order $n=p_1 p_2 p_3$, which Kohl denotes by $C_{mp}$, $C_p \times
D_q$, $F \times C_2$, $C_q \times D_p$, $D_{pq}$, $\Hol(C_p)$, where
$q=p_2$ and $p=p_3$. These contribute $1$, $2$, $2(p_2-1)$, $2$, $4$,
$2(p_2-1)$ Hopf-Galois structures respectively. The total number of
Hopf-Galois structures is therefore $4p_2+5$.  These groups are
respectively those of cases 1, 2, 5, 3, 4, 6 in our Table
\ref{3-prime-table}. Putting $p_1=2$ in Table \ref{3-prime-table}, we
again get $4p_2+5$ for the total number of Hopf-Galois structures, and
the number of Hopf-Galois structures of each type shown in Table
\ref{3-prime-table} agrees with Kohl's values. Thus our results
recover the part of Kohl's result \cite[Theorem 5.1]{kohl-mp} relating
to cyclic extensions of degree $2pq=2p_2(2p_2+1)$.

Now let $p_1>2$ and $p_2 \equiv p_3 \equiv 1 \pmod{p_1}$ but $p_3 \not
\equiv 1 \pmod{p_2}$, corresponding to the 4th row (``no--yes--yes'')
of our Table \ref{3-prime-total}. The first row of the table in
\cite[Theorem 2.4]{kohl-ANT} shows that there are $p_1+2$ groups
$G$. (Note that the final column, headed $C_{p_3p_2} \rtimes_i
C_{p_1}$, corresponds to $p_1-1$ distinct isomorphism types, given by
$1 \leq i \leq p_1-1$.) Of these groups, one contributes one
Hopf-Galois structure, two contribute $2(p_1-1)$, and the rest
$4(p_1-1)$. Thus there are in total of $(2p_1-1)^2$ Hopf-Galois
structures. This agrees with our count in Table \ref{3-prime-total}
and the relevant cases, 1--4, in Table \ref{3-prime-table}. (The
restriction $p_1>2$ turns out to be irrelevant when the Galois group
is cyclic.)

\subsection{Four primes}

As a final example, we consider the case when $n=p_1 p_2 p_3 p_4$ is
the product of $4$ distinct primes, under the assumption that 
\begin{equation} \label{all-cong}
  p_j \equiv 1 \pmod{p_i} \mbox{ whenever } i <j.
\end{equation}
(Thus we have $p_1<p_2<p_3<p_4$.)

We record in Table \ref{4-prime} the number of isomorphism classes of
groups $G(d,e,k)$, and the number of Hopf-Galois structures of each
type, corresponding to each relevant factorisation $n=dgz$. 

\begin{table} 
\centerline{ 
\begin{tabular}{|ccc|c|c|c|} \hline
  $d$ & $g$ & $z$ &  \# groups & \# HGS per group \\ \hline
  $1$ & $1$ & $p_1 p_2 p_3 p_4$  & $1$ & $1$ \\ \hline
  $p_1$ & $p_2$ & $p_3 p_4$ & $1$ & $2(p_1-1)$ \\
  $p_1$ & $p_3$ & $p_2 p_4$ & $1$ & $2(p_1-1)$ \\
  $p_1$ & $p_4$ & $p_2 p_3$ & $1$ & $2(p_1-1)$ \\
  $p_1$ & $p_2 p_3$ & $p_4$ & $p_1-1$ & $4(p_1-1)$ \\
  $p_1$ & $p_2 p_4$ & $p_3$ & $p_1-1$ & $4(p_1-1)$ \\
  $p_1$ & $p_3 p_4$ & $p_2$ & $p_1-1$ & $4(p_1-1)$ \\
  $p_1$ & $p_2 p_3 p_4$ & $1$ & $(p_1-1)^2$ & $8(p_1-1)$ \\ \hline
  $p_2$ & $p_3$ & $p_1 p_4$ & $1$ & $2(p_2-1)$ \\
  $p_2$ & $p_4$ & $p_1 p_3$ & $1$ & $2(p_2-1)$ \\
  $p_2$ & $p_3 p_4$ & $p_1$ & $p_2-1$ & $4(p_2-1)$ \\ \hline
  $p_3$ & $p_4$ & $p_1 p_2$ & $1$ & $2(p_3-1)$ \\ \hline
  $p_1 p_2$ & $p_3$ & $p_4$ & $1$ & $2(p_1-1)(p_2-1)$ \\
  $p_1 p_2$ & $p_4$ & $p_3$ & $1$ & $2(p_1-1)(p_2-1)$ \\
  $p_1 p_2$ & $p_3 p_4$ & $1$ & $(p_1+1)(p_2+1)-2$ & $4(p_1-1)(p_2-1)$
  \\ \hline
  $p_1 p_3$ & $p_4$ & $p_2$ & $1$ & $2(p_1-1)(p_3-1)$ \\ \hline
  $p_2 p_3$ & $p_4$ & $p_1$ & $1$ & $2(p_2-1)(p_3-1)$ \\ \hline
  $p_1 p_2 p_3$ & $p_4$ & $1$ & $1$ & $2(p_1-1)(p_2-1)(p_3-1)$ \\ \hline
\end{tabular}
}  
\vskip5mm

\caption{Numbers of isomorphism types and Hopf-Galois structures for
  $n=p_1 p_2 p_3 p_4$.} \label{4-prime} 
\end{table}
It follows from this table that, under the assumption
(\ref{all-cong}), there are $p_1^2+p_1p_2+2p_1+2p_2+8$ isomorphism types of
groups of order $n=p_1 p_2 p_3 p_4$, and the total number of
Hopf-Galois structures is
$$ 4p_1^2 p_2^2 +8p_1^3 + 2p_1 p_2 p_3 - 16 p_1^2
 -6 p_1 p_2  +10p_1 -1. $$
For example, if $n=2\cdot 3 \cdot 7 \cdot 43=1806$, or more
generally, if $n=42p_4$ for any prime $p_4 \equiv 1 \pmod{42}$, then
a cyclic extension of degree $n$ admits precisely $211$ Hopf-Galois
structures of $28$ different types. 

When (\ref{all-cong}) does not hold, we can enumerate the Hopf-Galois
structures by picking out the appropriate rows in Table \ref{4-prime},
just as we did in \S\ref{3-primes}. 

\bibliography{CSQ-bib}

\end{document}